\documentclass{article}      

\usepackage{graphicx}
\usepackage{amsmath}
\usepackage{hyperref}
\usepackage{url}
\usepackage{amssymb}
\usepackage{amsthm}

\newtheorem{theorem}{Theorem}[section]

\newtheorem{lemma}[theorem]{Lemma}
\newtheorem{proposition}[theorem]{Proposition}

\newtheorem{remark}[theorem]{Remark}

\def\Cond{\mathop{\rm Cond}}

\def\diag{\mathop{\rm diag}}
\usepackage{graphicx}

\title{Optimal interval length for the collocation of the Newton basis}

\author{J. M. Carnicer      \and
          Y. Khiar \and 
          J. M. Pe{\~{n}}a\thanks{This work has been partially supported by the Spanish Research Grant MTM2015-65433-P (MINECO/FEDER), by Gobierno the Arag\'on and Fondo Social Europeo.} \\
\small Departamento de Matem\'{a}tica Aplicada/IUMA, Universidad de Zaragoza \\ 
\small 50009 Zaragoza, Spain \\
}

\begin{document}
\maketitle
\begin{abstract}
It is known that the Lagrange interpolation problem at equidistant nodes is ill-conditioned. We explore the influence of the interval length in the computation of divided differences of the Newton interpolation formula. Condition numbers are computed for lower triangular matrices associated to the Newton interpolation formula at equidistant nodes. We consider the collocation matrices $L$ and $P_L$ of the monic Newton basis and a normalized Newton basis, so that $P_L$ is the lower triangular Pascal matrix. In contrast to $L$, $P_L$ does not depend on the interval length, and we show that the Skeel condition number of the $(n+1)\times (n+1)$ lower triangular Pascal  matrix is $3^n$. The $\infty$-norm condition number of the collocation matrix $L$ of the monic Newton basis is computed in terms of the interval length. The minimum asymptotic growth rate is achieved for intervals of length 3.
\end{abstract}


\section{Introduction}
\label{intro}
Divided differences can be used for obtaining approximations of the derivatives of a function, leading to numerical differentation formulae. In order to study the stability of the computation of divided differences arising in the Newton interpolation formula at equidistant nodes, we consider the behavior of the Newton basis at the nodes by means of the corresponding lower triangular collocation matrix. This paper studies the conditioning of these matrices. In particular, the collocation matrix $L$ of the monic Newton basis and the lower triangular Pascal matrix $P_L$ are considered. The traditional  and the Skeel condition numbers are used. It is proved that the Skeel condition number of the $(n+1)\times (n+1)$ lower triangular Pascal matrix is $3^n$. The $\infty$-norm condition number of $L$ is also obtained and its asymptotic behavior in terms of the interval length is analyzed and we show that the optimal interval length is 3. Analogous results for the collocation of the monomial basis and the conditioning of Vandermonde matrices can be found in \cite{Gautschi75}.

This paper considers the propagation of errors in the computation of divided differences in interpolation problems with equidistant nodes. On the one hand, interpolation at equidistant nodes is unstable (see Section 5 of \cite{JATckp}), illustrated by the Runge phenomenon (cf. \cite{Runge}). On the other hand, interpolation at equidistant nodes arises in practice when dealing with experimental data, tables of functions, difference equations and numerical integration with fixed stepsize. Moreover, the analysis of interpolation with equidistant nodes is a classical issue in Approximation Theory (cf. \S 4 of \cite{Steffensen}). As a consequence of the instability of the Lagrange interpolation problem at equidistant nodes, the computation of the divided differences is also unstable. We shall study the influence of the scaling and the interval length on this instability.

Section \ref{sec:2} presents basic concepts and notations as well as auxiliary results. The lower triangular matrices $L$ and $P_L$ are related with the computation of the divided and finite differences corresponding to two different forms of the Newton formula, based on different scaling. In Section \ref{sec:3}, the Skeel condition number of the lower triangular Pascal matrix $P_L$ is obtained. In Section \ref{sec:4}, $\kappa_\infty(L)$, the $\infty$-norm condition number of $L$, is given in terms of the interval length. Numerical examples are included. The asymptotic behavior of $\kappa_\infty(L)$, as the degree of the interpolant tends to infinity, is analyzed in Section \ref{sec:5}. It is shown that the interval length corresponding to a minimum asymptotic growth rate equals 3. Comparisons with the asymptotic behavior of the conditioning of $P_L$ are also performed.

\section{Basic notations and auxiliary results}
\label{sec:2}
Condition numbers measure the sensitivity of the solution of a linear system with respect to the perturbations of the data. For a given  matrix $A=(a_{ij})_{i,j=0,\dots,n}$, we shall denote by $|A|:=(|a_{ij}|)_{i,j=0,\dots,n}$ the matrix whose entries are the absolute values of the corresponding entries of $A$. 

The \emph{Skeel condition number} of a nonsingular matrix $A$ is given by
\begin{equation}\label{skeelcondition}
\Cond(A):=||\,|A^{-1}|\,|A|\,||_\infty.
\end{equation}

The usual $\infty$-norm condition number of a nonsingular matrix is defined as
\begin{equation}\label{conditionnumber}
\kappa_\infty(A):=||A||_\infty||A^{-1}||_\infty.
\end{equation}
By the submultiplicative property of the $\infty$-norm, we derive 
\begin{equation}\label{condkappa}
\Cond(A)\leq\kappa_\infty(A),
\end{equation}
so that the Skeel condition number gives lower bounds than the traditional condition number. Another feature of the Skeel condition number is that it is invariant under row scaling (see Section 7.2 of \cite{HighamBook}).

We want to analyze the conditioning of linear systems arising in the polynomial Lagrange interpolation problem. Given a function $f\in C[a,b]$ and distinct interpolation nodes $x_0,\dots,x_n$, there exists a unique polynomial $p$ in $P_n$, the space of polynomials with degree not greater than $n$,  such that $p(x_i)=f(x_i)$, $i=0,\dots,n$, called the Lagrange interpolation polynomial. 

The coefficients of the interpolation polynomial with respect to a basis can be computed by solving a linear system of equations, where some computations can be performed with high relative accuracy  (see \cite{MarcoMartinez}). An explicit solution of the polynomial interpolant is given by the Lagrange interpolation formula. In particular, the barycentric form of the Lagrange formula is recommended due to its computational advantages (cf. \cite{Berrut}). The Lagrange interpolation polynomial can also be expressed by means of the Newton formula. The  nice properties of the Newton formula are well-known. For instance, it provides a correction of the interpolation when the number of data increases by adding simple terms where divided differences play an essential role. This property of the Newton formula is used to estimate practical error bounds. 

The Newton formula is given by
\begin{equation}\label{Newtondivided}
p(x)=\sum_{i=0}^n d_if\,\omega_i(x)
\end{equation}
where
\begin{equation}\label{difdiv}
d_if:=[x_0,\dots,x_i]f,\quad i=0,\dots,n,
\end{equation}
are the divided difference functionals and
\begin{equation}\label{omega}
\omega_0(x):=1,\qquad \omega_i(x):=(x-x_0)\cdots(x-x_{i-1}),\quad i=1,\dots,n+1.
\end{equation}
is the monic Newton basis. The coefficients are the divided differences, which play a crucial role in numerical differentation. Another form of the Newton formula, based on different scaling, is given by
\begin{equation}\label{Newtonfinite}
p(x)=\sum_{i=0}^n\tilde d_if\tilde\omega_i(x),
\end{equation}
where
\begin{equation}\label{diferenciafinita}
\tilde d_if:=\omega_i(x_i)d_if,
\end{equation}
are the finite difference functionals and 
\begin{equation}\label{omegatilda}
\tilde\omega_i(x)=\frac{\omega_i(x)}{\omega_i(x_i)},\quad i=0,\dots,n,
\end{equation}
is a normalized Newton basis, in the sense that $\tilde\omega_i(x_i)=1$ for all $i=0,\dots,n$.

Since the nodes are distinct, we have the following explicit formulae for the divided differences
\begin{equation}\label{divided-differences}
d_kf=\sum_{i=0}^k \frac{f(x_i)}{\omega_{k+1}'(x_i)},
\end{equation}
with 
\begin{equation}\label{omegaprima}
\omega_{k+1}'(x_i)=\prod_{j\in\{0,\dots,k\}\setminus\{i\}}(x_i-x_j).
\end{equation}
For the finite differences we have 
\begin{equation}\label{finite-differences}
\tilde d_kf=\sum_{i=0}^k\frac{\omega_k(x_k)}{\omega_{k+1}'(x_i)}f(x_i).
\end{equation}

We are concerned with the conditioning of the problem of computing the divided differences (resp. finite differences) for a given function $f$. We shall adopt a matrix approach. Let us define the vectors 
\begin{equation}\label{defvectors}
{\bf d}:=(d_0f,\dots,d_nf)^T,\quad \tilde{\bf  d}:=(\tilde d_0f,\dots,\tilde d_nf)^T,\quad {\bf f}:=(f(x_0),\dots,f(x_n))^T,
\end{equation}
and the collocation matrices 
\begin{equation}\label{f19}
L=(\omega_j(x_i))_{0\leq i,j\leq n},\quad \tilde L=(\tilde\omega_j(x_i))_{0\leq i,j\leq n}.
\end{equation}
Taking into account that $\omega_k(x_i)=0$ for $k>i$, we deduce that $L$ and $\tilde L$ are lower triangular matrices. Observe that the matrix $\tilde L$ has ones on the diagonal. Besides the matrix $\tilde L$ is invariant under affine transformation of the nodes because, by (\ref{omega}) and (\ref{omegatilda}),
\begin{equation}\label{prodomegatilda}
\tilde\omega_j(x_i)=\frac{\omega_j(x_i)}{\omega_j(x_j)}=\prod_{k=0}^{j-1}\frac{x_i-x_k}{x_j-x_k}.
\end{equation}
is a product of simple ratios of the nodes. 

The sensitivity of divided  differences has been analyzed by several authors in different contexts (cf. Section 5.3 and 5.5 of \cite{HighamBook}). From (\ref{Newtondivided}) and (\ref{Newtonfinite}), ${\bf d}$ and $\tilde {\bf d}$, the coefficients of the Newton formulae, are the solutions of the systems
\begin{equation}\label{system}
L {\bf d}={\bf f},\quad \tilde L \tilde{\bf d}={\bf f},
\end{equation}
respectively. On the other hand, $L$ (resp., $\tilde L$) is the matrix of change of basis between the Lagrange basis and the monic (resp.,  normalized) Newton basis. Thus, we are interested in the computation of the condition numbers of these matrices. Note that if the component $d_i$ of the vector of divided differences is computed with high relative error, this can be compensated if the corresponding factor $\omega_i(x)$ is sufficiently small. In practice, inaccurate computation of divided differences may still reproduce the interpolation polynomial well (see page 100 of \cite{HighamBook}).

From the system (\ref{system}), we obtain $L^{-1}{\bf f}={\bf d}$, and using (\ref{divided-differences}), we conclude that the entries of $L^{-1}$ are
\begin{equation}\label{inversaL}
l_{kj}^{(-1)}=\begin{cases}0,&\text{if } j>k, \\
 \frac{1}{\omega_{k+1}'(x_j)},& \text{if } j\leq k.\end{cases}
\end{equation}

From formula (\ref{omegatilda}), we obtain the relation between $L$ and $\tilde L$
\begin{equation}\label{relation}
\tilde L=LD, \quad D=\diag(1/\omega_0(x_0),\dots,1/\omega_n(x_n)).
\end{equation}

From now on we will consider equidistant nodes $x_0,\dots,x_n$ in an interval $[a,b]$ in increasing order, that is,
\begin{equation}\label{equidistant}
x_i=a+(b-a)\frac{i}{n},\quad i=0,\dots,n.
\end{equation}
In this particular case, by (\ref{omega}), the entries of $L=(l_{ij})_{0\leq i,j\leq n}$ of (\ref{f19}) for $j\leq i$ are given by
\begin{equation}\label{omegaeq}
	l_{ij}=\omega_j(x_i)=\prod_{k=0}^{j-1} (x_i-x_k)=\Big(\frac{b-a}{n}\Big)^j\prod_{k=0}^{j-1}(i-k)=\Big(\frac{b-a}{n}\Big)^j\frac{i!}{(i-j)!}.
\end{equation}

By (\ref{omegaprima}) and (\ref{inversaL}), we also have for $j\leq i$
\begin{equation}\label{omegaprimaeq}
l_{ij}^{(-1)}=\Big(\frac{n}{b-a}\Big)^i\frac{1}{\prod_{k\in\{0,\dots,i\}\setminus\{j\}}(j-k)}=(-1)^{i-j}\Big(\frac{n}{b-a}\Big)^i\frac{1}{j!(i-j)!}.
\end{equation}

\section{Conditioning of  Pascal matrices}
\label{sec:3}
Recall that the lower triangular Pascal matrix $P_L:=(q_{ij})_{0\leq i,j\leq n}$ is given by (cf. \cite{AlonsoPena})
\begin{equation}\label{pascalq}
q_{ij}:=\dbinom{i}{j}.
\end{equation}
Pascal matrices play an important role in many fields (cf. \cite{Edelman},\cite{Lewis}) and its well-known (cf. \cite{Chen} and Example 6.1 of Chapter 3 of \cite{Karlin}) that they are totally positive  matrices, that is, all their minors are nonnegative.

\begin{remark}\label{remarkpascal}
For equidistant nodes (\ref{equidistant}), using (\ref{prodomegatilda}), we have that the entries of the collocation matrix $\tilde L=(\tilde\omega_j(x_i))_{0\leq i,j\leq n}$ are
\begin{equation}\label{entriesPL}
\tilde\omega_j(x_i)=\prod_{k=0}^{j-1}\frac{i-k}{j-k}=\frac{i(i-1)\cdots(i-j+1)}{j(j-1)\cdots1}=\dbinom{i}{j}.
\end{equation}
We conclude that the collocation matrix associated to the Newton representation with finite differences $\tilde L$ does not depend on the interval $[a,b]$ and coincides with the lower triangular Pascal matrix $P_L$, that is, $\tilde L=P_L$.
\end{remark}

By the relation (\ref{relation}), we have
\begin{equation}\label{lpl}
P_L^{-1}=D^{-1}L^{-1}.
\end{equation}
So, using  (\ref{omegaeq}) and (\ref{omegaprimaeq}) and the previous relation, the entries of the matrix $P_L^{-1}=(q_{ij}^{(-1)})_{0\leq i,j\leq n}$ are, for $j\leq i$,
\begin{equation}\label{entriesPLinv}
q_{ij}^{(-1)}=\frac{\omega_i(x_i)}{\omega'_{i+1}(x_j)}=(-1)^{i+j}\dbinom{i}{j}.
\end{equation}

In the following result we compute the Skeel condition number of the lower triangular Pascal matrix and of its inverse.

\begin{theorem}\label{prop8}
Let $P_L$ be the lower triangular Pascal matrix. Then
$$
\Cond(P_L)=\Cond(P_L^{-1})=3^n.
$$
\end{theorem}
\begin{proof}
By (\ref{pascalq}) and (\ref{entriesPLinv}), $|P_L|=|P_L^{-1}|$. Then we have
$$
\Cond(P_L)=||\,|P_L^{-1}|\,|P_L|\,||_\infty=||\,|P_L|\,|P_L^{-1}|\,||_\infty=\Cond(P_L^{-1}).
$$
By (\ref{entriesPLinv}) and (\ref{pascalq}), we can compute the Skeel condition number
\begin{align*}
&\Cond(P_L)=||\,|P_L^{-1}|\,|P_L|\,||_\infty=\max_{i=0,\dots,n}\sum_{j=0}^i\sum_{k=j}^i\dbinom{i}{k}\dbinom{k}{j}\\
&=\max_{i=0,\dots,n}\sum_{j=0}^i\sum_{k=j}^i\dbinom{i}{j}\dbinom{i-j}{k-j}
=\max_{i=0,\dots,n}\sum_{j=0}^i\dbinom{i}{j}2^{i-j}
=\max_{i=0,\dots,n}3^i=3^n.
\end{align*}\qed

\end{proof}
Let us observe that Theorem \ref{prop8} provides a lower bound for $\kappa_\infty(P_L)$ and $\kappa_\infty(L)$. In fact, by (\ref{condkappa}),
$$
3^n=\Cond(P_L)\leq \kappa_\infty(P_L),
$$
and, by (\ref{skeelcondition}), (\ref{lpl}) and (\ref{conditionnumber}), 
\begin{equation}\label{3n}
3^n=\Cond(P_L^{-1})=||\, |LD|\,|D^{-1}L^{-1}|\,||_\infty=\Cond(L^{-1})\leq \kappa_\infty(L^{-1})=\kappa_\infty(L).
\end{equation}
From (\ref{pascalq}) and (\ref{entriesPLinv}), it follows that $||P_L||_\infty=||P_L^{-1}||_\infty=2^n$. Then we can state the following known result (cf. Proposition 2 of \cite{PenaKhiarCarni}). Related inequalities can also be derived using the analysis of the spectral conditioning of a Pascal matrix given in page 520 of \cite{HighamBook}.
\begin{proposition}\label{pascal4n}
The $\infty$-norm condition number of the lower triangular Pascal matrix is $\kappa_\infty(P_L)=4^n$.
\end{proposition}

\section{Condition number of the collocation matrix of Newton basis}
\label{sec:4}
In this section we are going to study the behavior of the matrix $L$ at equidistant nodes given by (\ref{equidistant}), whose entries are computed in  (\ref{omegaeq}). This matrix is not invariant by affine transformations of the nodes and its entries depend on the length of the interval $\ell:=b-a$. 

Let us recall the upper incomplete gamma function
$$
\Gamma(a,x):=\int_x^{+\infty}t^{a-1}e^{-t}\, dt.
$$
From the definition, we deduce that 
$$
\Gamma(a,x)<\Gamma(a,0)=:\Gamma(a), \quad \forall x>0.
$$
It is well-known (see  formula 8.4.8 of \cite{Handbook})  that the incomplete gamma function gives an integral representation of the Taylor polynomial of the exponential function
\begin{equation}\label{l7}
\sum_{k=0}^n\frac{x^k}{k!}=e^x\frac{\Gamma(n+1,x)}{\Gamma(n+1)}=\frac{1}{n!}\int_0^{+\infty} e^{x-t}t^n\, dt.
\end{equation}

The following result gives us an inequality for this function.
\begin{lemma}\label{l8} If $\ell>1$, then
$$
\Gamma(n+1,\frac{n}{\ell})\geq \Gamma(n+1)-e^{-n/\ell}\Big(\frac{n}{\ell}\Big)^{n+1}
$$
and
$$
\lim_{n\rightarrow\infty}\frac{\Gamma(n+1,n/\ell)}{\Gamma(n+1)}=1.
$$
\end{lemma}
\begin{proof}
If $\ell>1$ we have
$$
0<\frac{\Gamma(n+1)-\Gamma(n+1,n/\ell)}{\Gamma(n+1)}=\int_0^{n/\ell}\frac{t^n}{n!}e^{-t}\, dt.
$$
The function $f(t)=t^ne^{-t}$ is increasing for $0\leq t\leq n$. So
\begin{equation*}
\frac{\Gamma(n+1)-\Gamma(n+1,n/\ell)}{\Gamma(n+1)}\leq\int_0^{n/\ell}\Big(\frac{n}{\ell}\Big)^n\frac{1}{n!}e^{-n/\ell}\, dt=\Big(\frac{n}{\ell}\Big)^{n+1}\frac{1}{n!}e^{-n/\ell},
\end{equation*}
and we deduce the inequality
$$
\Gamma(n+1,\frac{n}{\ell})\geq \Gamma(n+1)-e^{-n/\ell}\Big(\frac{n}{\ell}\Big)^{n+1}.
$$
Hence, in order to prove the result, it is sufficient to see that $(n/\ell)^{n+1}\frac{1}{n!}e^{-n/\ell}\to 0$ as $n\to\infty$. We denote by
$$
c_n:=\Big(\frac{n}{\ell}\Big)^{n+1}\frac{1}{n!}e^{-n/\ell}.
$$
We have
\begin{align*}
\frac{c_{n+1}}{c_n}&=\Big(\frac{n+1}{\ell}\Big)^{n+2}\frac{1}{(n+1)!}e^{-(n+1)/\ell}\Big(\frac{n}{\ell}\Big)^{-(n+1)}n!e^{n/\ell}\\
&=\Big(\frac{n+1}{n}\Big)^{n+1}\frac{n+1}{\ell}\frac{1}{n+1}e^{-1/\ell}\to \frac{1}{\ell}e^{1-1/\ell},\quad \text{ as } n\to\infty.
\end{align*}
Let us show that $  e^{1-1/\ell}/\ell<1$. Let be $g(x):=xe^{1-x}$. Then
$$
g'(x)=e^{1-x}(1-x)>0,\quad \text{ if } x<1.
$$
Hence, $g(x)$ is increasing for $x<1$, and thus
$$
\frac{1}{\ell}e^{1-1/\ell}=g\Big(\frac{1}{\ell}\Big)<g(1)=1.
$$
So, $\lim_{n\to\infty}c_{n+1}/c_n< 1$ and  $\lim_{n\to\infty} c_n=0$.\qed

\end{proof}

The following result provides $||L||_\infty$. 
\begin{proposition}\label{prop9}
Let $L$ be the lower triangular matrix given by (\ref{f19}) at equidistant nodes in $[a,b]$ given by (\ref{equidistant}) and let $\ell=b-a$. Then
$$
||L||_\infty=n!\Big(\frac{n}{\ell}\Big)^{-n}\sum_{k=0}^n\frac{1}{k!}\Big(\frac{n}{\ell}\Big)^{k}=\Big(\frac{n}{\ell}\Big)^{-n}e^{n/\ell}\Gamma\Big(n+1,\frac{n}{\ell}\Big).
$$
\end{proposition}
\begin{proof}
By (\ref{f19}), the $\infty$-norm of $L$ is
$$
||L||_\infty=\max_{i=0,\dots,n}\sum_{k=0}^i|\omega_k(x_i)|,
$$
and taking into account that
$$
|\omega_k(x_i)|\leq|\omega_k(x_n)|,\quad k=0,\dots,i,\quad i=0,\dots,n,
$$
we have that this maximum is achieved in $n$. Using (\ref{omegaeq}), we derive
\begin{equation}\label{f23}
||L||_\infty=\sum_{k=0}^n|\omega_k(x_n)|=n!\sum_{k=0}^n\frac{1}{(n-k)!}\Big(\frac{\ell}{n}\Big)^k.
\end{equation} 
By formula  (\ref{f23})
\begin{equation*}
||L||_\infty=n!\sum_{k=0}^n\frac{1}{(n-k)!}\Big(\frac{\ell}{n}\Big)^k
=n!\sum_{k=0}^n\frac{1}{k!}\Big(\frac{\ell}{n}\Big)^{n-k}=n!\Big(\frac{n}{\ell}\Big)^{-n}\sum_{k=0}^n\frac{1}{k!}\Big(\frac{n}{\ell}\Big)^{k}.
\end{equation*}
Using formula (\ref{l7}), we obtain the result.
\qed
\end{proof}
From (\ref{f23}), we deduce that the $\infty$-norm of $L$ is an increasing function of the interval length for a given value of $n$.

The computation of the $\infty$-norm of $L^{-1}$ has different cases depending on the interval length. For this purpose, we use the floor function
$$
\lfloor x \rfloor:=\max\{k\in \mathbb{Z} | k\leq x\}.
$$
\begin{proposition}\label{prop10}
Let $L$ be the lower triangular matrix given by (\ref{f19}) at equidistant nodes in $[a,b]$ given by (\ref{equidistant}) and let $\ell=b-a$. Then
\begin{equation*}
||L^{-1}||_\infty=\begin{cases}\frac{1}{n!}\Big(\frac{2n}{\ell}\Big)^n,& \text{ if } \ell\leq 2,\\
							\frac{1}{i_n!}\Big(\frac{2n}{\ell}\Big)^{i_n},\quad i_n=\lfloor \frac{2n}{\ell} \rfloor,& \text{ if }2\leq \ell \leq 2n,\\
							1,& \text{ if }  \ell\geq2n.\end{cases}
\end{equation*}
\end{proposition}
\begin{proof}
By formula (\ref{inversaL}),
\begin{equation}\label{normainversaL}
||L^{-1}||_\infty= \max_{i=0,\dots,n}\left(\frac{n}{\ell}\right)^i\sum_{j=0}^i\frac{1}{j!(i-j)!}= \max_{i=0,\dots,n}\frac{1}{i!}\Big(\frac{2n}{\ell}\Big)^i.
\end{equation}
We define
$$
r_i:=\frac{1}{i!}\Big(\frac{2n}{\ell}\Big)^i,\quad i=0,\dots,n.
$$
Let us compute the maximum of the sequence $\{r_i\}_{i=0,\dots,n}$. We consider the quotient between two consecutive elements
$$
\frac{r_{i+1}}{r_i}=\frac{2n}{(i+1)\ell}.
$$
If $2/\ell\geq 1$  the sequence is increasing and then
	$$
	||L^{-1}||_\infty=r_n=\frac{1}{n!}\Big(\frac{2n}{\ell}\Big)^n.
	$$
	
If $2n/\ell\leq 1$, the sequence is decreasing and so
	$$
	||L^{-1}||_\infty=r_0=1.
	$$
	
Finally, if $1\leq2n/\ell\leq n$ the maximum is achieved at $i_n:=\lfloor \frac{2n}{\ell} \rfloor$, that is,
	$$
	||L^{-1}||_\infty=r_{i_n}=\frac{1}{i_n!}\Big(\frac{2n}{\ell}\Big)^{i_n}.
	$$\qed

\end{proof}

As a consequence of the previous propositions \ref{prop9} and \ref{prop10} we obtain the following result for $\kappa_\infty(L)$.
\begin{theorem}\label{prop11}
Let $L$ be the lower triangular matrix given by (\ref{f19}) at equidistant nodes in $[a,b]$ given by (\ref{equidistant}) and let $\ell=b-a$. Then
\begin{equation*}
\kappa_\infty(L)=\begin{cases}2^n\sum_{k=0}^n\frac{1}{k!}\Big(\frac{n}{\ell}\Big)^k,& \ell\leq 2,\\
							\frac{n!}{i_n!}2^{i_n}\Big(\frac{\ell}{n}\Big)^{n-i_n}\sum_{k=0}^n\frac{1}{k!}\Big(\frac{n}{\ell}\Big)^k,\quad i_n=\lfloor \frac{2n}{\ell} \rfloor,& 2\leq \ell \leq 2n,\\
							n!\Big(\frac{\ell}{n}\Big)^n\sum_{k=0}^n\frac{1}{k!}\Big(\frac{n}{\ell}\Big)^k,&  \ell\geq2n.\end{cases}
\end{equation*}
\end{theorem}

Let us analyze some consequences of the previous result. For $\ell\leq2$, we have that $\kappa_\infty(L)$ is a decreasing function of the interval length. So, in this case, the lowest conditioning is attained at $\ell=2$ and its value can be bounded by
$$
\kappa_\infty(L)=2^n \sum_{k=0}^n\frac{1}{k!}\Big(\frac{n}{2}\Big)^k\leq (2\sqrt e)^n,\quad \ell=2.
$$
If  $\ell\geq 2n$ then $\kappa_\infty(L)$ is an increasing function of $\ell$ and its smallest value is obtained when $\ell=2n$
$$
\kappa_\infty(L)= n!2^n\sum_{k=0}^n\frac{1}{k!}2^{-k}\geq n!2^n, \quad \ell=2n.
$$
Since $n!\geq e^{n/2}$ for $n\geq3$, we have that  $\kappa_\infty(L)$ is higher for $\ell\geq 2n$ than for $\ell=2$. Furthermore, taking into account the growth more than exponential of the factorial, we show that the conditioning increases much more than in the case $\ell=2$. So, in order to have low condition number, we must take interval lengths satisfying $2\leq \ell \leq 2n$.

Table \ref{table1} shows the conditioning of the matrix $L$ at equidistant nodes for different values of $n$. We have analyzed $\kappa_\infty(L)$ in the intervals $[0,1]$, $[0,2]$ y $[0,3]$, with respective lengths 1, 2 and 3. We also see that the intervals of lengths 2  and 3 give better results than the standard interval $[0,1]$. 

\begin{table}
\caption{$\kappa_\infty(L)$ at equidistant nodes in different intervals.}
\label{table1}       
\begin{tabular}{llll}
\hline\noalign{\smallskip}
n&$\kappa_\infty(L)$ on $[0,1]$&$\kappa_\infty(L)$ on $[0,2]$&$\kappa_\infty(L)$ on $[0,3]$ \\
\noalign{\smallskip}\hline\noalign{\smallskip}
3&104&$33.5$&32\\ 
4&$549.3333$&112&$101.2222$\\
5&$2.9253\times10^3$&$373.4583$&$302.7358$\\ 
9&$2.4370\times10^6$&$4.5301\times10^4$&$2.3969\times10^4$\\ 
14&$1.1239\times10^{10}$&$1.7865\times10^7$&$5.9094\times10^6$\\ 
19& $5.2459\times10^{13}$&$6.9906\times10^9$&$1.4329\times10^9$\\ 
\noalign{\smallskip}\hline
\end{tabular}
\end{table}

			

\section{Asymptotic analysis of condition number of the collocation matrix of Newton basis}
\label{sec:5}
In this section we want to analyze whether there exists an interval length $\ell=b-a$ such that the growth of $\kappa_\infty(L)$ is as small as possible. We will show that  $\lim_{n\to\infty}\kappa_\infty(L)^{1/n}$ is a constant and, by (\ref{3n}), this constant is greater than or equal to 3. Therefore,  $\kappa_\infty(L)$ presents an exponential growth. We will also show that the length corresponding to a minimum asymptotic growth rate is $ \ell=3$.

\begin{theorem}\label{t9}
Let $L$ be the lower triangular matrix given by (\ref{f19}) at equidistant nodes in $[a,b]$ given by (\ref{equidistant}) and let $\ell=b-a$. Then
\begin{equation*}
\lim_{n\to\infty} \kappa_\infty(L)^{1/n}=\begin{cases}\ell e^{3/\ell-1}, & \ell\geq2,\\
2e^{1/\ell}, & 1< \ell\leq2,\\
\frac{2e}{\ell}, & \ell\leq1.\end{cases}
\end{equation*}
The lowest value of the previous limit is $3$ for $\ell=3$ and we have
\begin{equation}\label{minrate}
\lim_{n\to\infty}\frac{ \kappa_\infty(L)}{3^n}=\sqrt\frac{3}{2},\quad \ell=3.
\end{equation}
\end{theorem}
\begin{proof}
If $\ell\geq2$ there exists a sufficiently large $n$ such that $2\leq \ell\leq 2n$. Let $i_n:=\lfloor \frac{2n}{\ell} \rfloor$. Using formula (\ref{l7}), we deduce from Theorem \ref{prop11} that
\begin{equation}\label{f27}
\kappa_\infty(L)=\ell^n \frac{n!n^{i_n}}{i_n!n^n}\Big(\frac{2}{\ell}\Big)^{i_n}e^{n/\ell}\frac{\Gamma(n+1,\frac{n}{\ell})}{n!}=\Big(\ell e^{1/\ell}\Big)^n\Big(\frac{2}{\ell}\Big)^{i_n}\frac{n!n^{i_n}}{i_n!n^n}\frac{\Gamma(n+1,\frac{n}{\ell})}{n!}.
\end{equation}
Since
\begin{equation}\label{inn}
\lim_{n\to\infty}\frac{i_n}{n}=\lim_{n\to\infty}\frac{\lfloor \frac{2n}{\ell} \rfloor}{n}=\frac{2}{\ell},
\end{equation}
we have
\begin{equation*}
\lim_{n\to\infty}\Big(\frac{n^{i_n}}{i_n!}\Big)^{1/n}=\lim_{n\to\infty}\Big(\frac{n^{i_n}}{e^{-i_n}\sqrt{2\pi i_n}i_n^{i_n}}\Big)^{1/n}=\lim_{n\to\infty}\Big(\frac{e}{i_n/n}\Big)^{i_n/n}=\Big(\frac{\ell e}{2}\Big)^{2/\ell}.
\end{equation*}
We also have
$$
\lim_{n\to\infty}\Big(\frac{n!}{n^n}\Big)^{1/n}=e^{-1}.
$$
Applying Lemma \ref{l8}, we deduce that
$$
\lim_{n\to\infty}\frac{\Gamma(n+1,\frac{n}{\ell})}{n!}=1.
$$
So, for $\ell\geq 2$
\begin{equation*}
\lim_{n\to\infty}\kappa_\infty(L)^{1/n}= \ell e^{1/\ell}\Big(\frac{2}{\ell}\Big)^{2/\ell}e^{-1}\Big(\frac{\ell e}{2}\Big)^{2/\ell}=\ell e^{3/\ell-1}.
\end{equation*}
If $\ell\leq 2$, we have by formula (\ref{l7}) and Theorem \ref{prop11}
\begin{equation*}
\kappa_\infty(L)=2^n\sum_{k=0}^n\frac{1}{k!}\Big(\frac{n}{\ell}\Big)^k=\Big(2e^{1/\ell}\Big)^n\frac{\Gamma(n+1,\frac{n}{\ell})}{n!}.
\end{equation*}
If $1<\ell\leq2$, we can use Lemma \ref{l8} and deduce that
\begin{equation*}
\lim_{n\to\infty}\kappa_\infty(L)^{1/n}=2e^{1/\ell}.
\end{equation*}

In the case $\ell<1$, we have 
\begin{align*}
\kappa_\infty(L)&=2^n\sum_{k=0}^n\frac{1}{k!}\Big(\frac{n}{\ell}\Big)^k=2^n\frac{n^n}{n!\ell^n}\Big(1+\frac{n\ell}{n}+\frac{n(n-1)\ell^2}{n^2}+\dots+\frac{n!\ell^n}{n^n}\Big)\\
&\leq 2^n\frac{n^n}{n!\ell^n}\Big(1+\ell+\dots+\ell^n\Big)=2^n\frac{n^n}{n!\ell^n}\frac{1-\ell^{n+1}}{1-\ell}\leq2^n\frac{n^n}{n!\ell^n}\frac{1}{1-\ell}.
\end{align*}
Taking the limit as $n\to\infty$ of the $n$-th root,
\begin{equation*}
\lim_{n\to\infty}\sup\kappa_\infty(L)^{1/n}\leq\lim_{n\to\infty}\Big(2^n\frac{n^n}{n!\ell^n}\frac{1}{1-\ell}\Big)^{1/n}=\frac{2}{\ell}\lim_{n\to\infty}\Big(\frac{n^n}{n!}\Big)^{1/n}=\frac{2e}{\ell}.
\end{equation*}

For $\ell=1$,
\begin{equation*}
\kappa_\infty(L)=2^n\sum_{k=0}^n\frac{n^k}{k!}=2^n\frac{n^n}{n!}\Big(1+\frac{n}{n}+\frac{n(n-1)}{n^2}+\dots+\frac{n!}{n^n}\Big)
\end{equation*}
and, since each term of the sum inside the brackets is less than or equal to 1, we have
\begin{equation*}
\lim_{n\to\infty}\sup\kappa_\infty(L)^{1/n}\leq\lim_{n\to\infty}\Big(2^n\frac{n^{n}(n+1)}{n!}\Big)^{1/n}=2e.
\end{equation*}
On the other hand,
$$
\kappa_\infty(L)=2^n\sum_{k=0}^n\frac{1}{k!}\Big(\frac{n}{\ell}\Big)^k\geq\Big(\frac{2n}{\ell}\Big)^n\frac{1}{n!},\quad \ell\leq1.
$$
Hence, for $\ell\leq 1$
$$
\lim_{n\to\infty}\inf\kappa_\infty(L)^{1/n}\geq\lim_{n\to\infty}\Big[\Big(\frac{2n}{\ell}\Big)^n\frac{1}{n!}\Big]^{1/n}=\frac{2}{\ell}\lim_{n\to\infty}\Big(\frac{n^n}{n!}\Big)^{1/n}=\frac{2e}{\ell}.
$$
Therefore
$$
\lim_{n\to\infty}\kappa_\infty(L)^{1/n}=\frac{2e}{\ell}.
$$

Let us observe that, for $\ell\leq2$, $\lim_{n\to\infty}\kappa_\infty(L)^{1/n}$ is a decreasing function of $\ell$. Since the function $ \ell e^{3/\ell-1}$ attains its minimum at $\ell=3$, we have the lowest exponential growth for $\ell=3$. In this case, using formula (\ref{f27}) with $\ell=3$, Lemma \ref{l8} and the Stirling's formula (see formula 5.11.7 of \cite{Handbook}), we obtain

\begin{align*}
&\lim_{n\to\infty}\frac{\kappa_\infty(L)}{3^n}=\lim_{n\to\infty}\frac{1}{3^n}\frac{3^nn!n^{i_n}}{i_n!n^n}\Big(\frac{2}{3}\Big)^{i_n} \frac{\Gamma(n+1,\frac{n}{3})}{n!}e^{n/3}\\
&=\lim_{n\to\infty}\frac{n^ne^{-n}\sqrt{2\pi n}n^{i_n}}{i_n^{i_n}e^{-i_n}\sqrt{2\pi i_n}n^n}\Big(\frac{2}{3}\Big)^{i_n} \frac{\Gamma(n+1,\frac{n}{3})}{n!}e^{n/3}=\lim_{n\to\infty}\Big(\frac{2n}{3i_n}\Big)^{i_n}e^{i_n-2n/3}\sqrt{\frac{n}{i_n}}.
\end{align*}
By formula (\ref{inn}), we have $i_n/n\to2/3$, as $ n\to\infty$. Let us denote by 
$$
s_n:=e^{i_n-2n/3}\Big(\frac{2n}{3i_n}\Big)^{i_n}
$$
and show that $\lim_{n\to\infty}s_n=1$ or, equivalently, $\lim_{n\to\infty}\log s_n=0$. We use the following property
$$
\lim_{x\to0}\frac{\log(1-x)+x}{x}=0.
$$
to derive that
$$
f(x):=\begin{cases}\frac{\log(1-x)+x}{x}, & x\neq0,\\
0, &x=0,
\end{cases}
$$
is continuous in $x=0$. Since $i_n-2n/3$ is bounded, we have
\begin{align*}
\lim_{n\to\infty}\log s_n&=\lim_{n\to\infty}\left(i_n-\frac{2n}{3}\right)+i_n\log\left(\frac{2n}{3i_n}\right)\\
&=\lim_{n\to\infty} f\left(1-\frac{2n}{3i_n}\right)\left(i_n-\frac{2n}{3}\right)=0.
\end{align*}
Hence, 
$$
\lim_{n\to\infty}\frac{\kappa_\infty(L)}{3^n}=\lim_{n\to\infty}s_n\lim_{n\to\infty}\sqrt{\frac{n}{i_n}}=\sqrt{\frac{3}{2}}.
$$
\qed
\end{proof}

Observe that for $\ell\leq 1$, the lowest value of  $\lim_{n\to\infty}\kappa_\infty(L)^{1/n}$  is attained in $\ell=1$ and this limit is $2e\approx 5.4366$. If $1\leq \ell\leq2$, the smallest value is $2\sqrt e\approx 3.2974$ for $\ell=2$. Finally, for $\ell\geq 2$ we have that the minimum is 3 for $\ell=3$. 

Let us compare $\kappa_\infty(L)^{1/n}$ and $\kappa_\infty(P_L)^{1/n}$ for different interval lengths. For  $1/\log2\leq\ell\leq 2$, we have
$$
\kappa_\infty^{1/n}(L)=2e^{1/\ell}\leq 4=\kappa_\infty(P_L)^{1/n}.
$$
If $\ell\geq2$ we observe that $\lim_{n\to\infty}\kappa_\infty(L)^{1/n}\leq\lim_{n\to\infty}\kappa_\infty(P_L)^{1/n}$ if and only if
$$
\ell e^{3/\ell-1}\leq 4.
$$
This inequality holds for lengths between 2 and approximately 7.1451. If $\ell\leq1$, $\lim_{n\to\infty}\kappa_\infty(L)^{1/n}=2e/\ell\geq 4=\kappa_\infty(P_L)^{1/n}$, that is, $L$ has worse asymptotic behavior than $P_L$.

We have proved that the minimum asymptotic growth rate of $\kappa_\infty(L)$ is achieved in intervals of length 3. In order to take advantage of the good properties of the intervals of length 3, we can perform an affine change of variables from the working interval $[a,b]$ to an interval of length 3 and compute the divided differences, $d_i^{(3)}f$, with respect to the  transformed nodes in  the  interval of length 3. In this case, divided differences are rescaled 
$$
d_i^{(3)}f=\Big(\frac{b-a}{3}\Big)^id_if,\quad i=0,\dots,n,
$$
which implies the following rescaling of the Newton basis
$$
\omega_i^{(3)}(x):=\Big(\frac{3}{b-a}\Big)^i\omega_i(x),\quad i=0,\dots,n.
$$
With these normalizations we gain stability in the processes of getting divided differences from data and recovering data  from the divided differences.

The matrix interpretation of this procedure is that the collocation matrix $L$ has to be replaced by the matrix $L^{(3)}=(\omega_j^{(3)}(x_i))_{i,j=0,\dots,n}$. Both matrices are related by
$$
L^{(3)}=L\diag\left(1,\frac{3}{b-a},\dots,\Big(\frac{3}{b-a}\Big)^n\right).
$$
By Theorem  \ref{t9}, the asymptotic condition number of $L^{(3)}$ is given by
$$
\kappa_\infty(L^{(3)})\sim \sqrt{\frac{3}{2}}\,3^n,
$$
providing a more stable alternative than finite differences.

%



\end{document}